\documentclass[12pt,leqno]{amsart}
\usepackage{amscd,times,fullpage}
\usepackage[utf8]{inputenc}
\usepackage{amsmath}
\usepackage{mathtools}
\usepackage{amsfonts}
\usepackage{amssymb}
\usepackage{booktabs}

\usepackage{amsthm}
\usepackage{tikz-cd}

\usepackage[hypertexnames=false,
backref=page,
    pdftex,
    pdfpagemode=UseNone,
    breaklinks=true,
    extension=pdf,
    colorlinks=true,
    linkcolor=blue,
    citecolor=blue,
    urlcolor=blue,
]{hyperref}

\let\OLDthebibliography\thebibliography  %condense bibliography
\renewcommand\thebibliography[1]{
	\OLDthebibliography{#1}
	\setlength{\parskip}{0pt}
	\setlength{\itemsep}{0pt plus 0.3ex}
}

\theoremstyle{plain}
\newtheorem{thm}{Theorem}
\newtheorem{cor}{Corollary}
\newtheorem{lem}{Lemma}

\theoremstyle{definition}

\newtheorem{rem}{Remark}

\theoremstyle{plain}
\newtheorem{thmintro}{Theorem}

\newtheorem{corintro}{Corollary}[thmintro]

\newcommand{\HH}{\mathbb{H}}

\newcommand{\C}{\mathbb{C}}
\newcommand{\Z}{\mathbb{Z}}

\newcommand{\cA}{\mathcal{A}}

\newcommand{\cL}{\mathcal{L}}

\newcommand{\cS}{\mathcal{S}}

\newcommand{\im}{\operatorname{im}}
\newcommand{\del}{\partial}
\newcommand{\delbar}{{\bar{\partial}}}

\newcommand{\pr}{\operatorname{pr}}

\newcommand{\supp}{\operatorname{supp}}

\DeclareMathOperator{\mult}{mult}

\newcommand{\img}[2][1]{\begin{gathered}\includegraphics[scale=#1]{#2}\end{gathered}}

\newcommand{\Cdot}{{\raisebox{-0.7ex}[0pt][0pt]{\scalebox{2.0}{$\cdot$}}}}
\mathchardef\mhyphen="2D

\title{Some remarks on the Schweitzer complex}

\author{Jonas Stelzig} 
\address{Mathematisches Institut der LMU M\"unchen, Theresienstra{\ss}e 39, 80333 M\"unchen.}
\email{jonas.stelzig@math.lmu.de}

\begin{document}

\begin{abstract}
We prove that the Schweitzer complex is elliptic and its cohomologies define cohomological functors. As applications, we obtain finite dimensionality, a version of Serre duality, restrictions of the behaviour of cohomology in small deformations, and an index formula which turns out to be equivalent to the Hirzebruch-Riemann-Roch relations.
\end{abstract}
\dedicatory{Dedicated to the memory of J.-P.~Demailly}
\maketitle
\section*{Introduction} The Bott-Chern and Aeppli cohomology are classical and well-established invariants of complex manifolds \cite{aeppli_exact_1962, bott_hermitian_1965}. Given a complex manifold $X$, with double complex of forms $(\cA_X(X),\del,\delbar)$ they are defined as
\[
H_{BC}^{p,q}(X):=\frac{\ker\del\cap\ker\delbar}{\im\del\delbar}\qquad\text{and}\qquad H_A^{p,q}(X):=\frac{\ker\del\delbar}{\im\del+\im\delbar}.
\]

M. Schweitzer and J.~P. Demailly \cite{schweitzer_autour_2007, demailly_complex_nodate} have shown that they arise as the cohomology groups in certain degrees of a differential complex $\cL_{p,q}:=(\cL_{p,q}^\Cdot(X),d_{\cL})$, the definition of which we recall below. The cohomologies of this complex appear naturally when considering higher-length Aeppli-Bott-Chern Massey products \cite{milivojevic_stelzig_2022}, but have received relatively little attention otherwise. In this short note, we will establish some basic properties of the Schweitzer complex. More precisely, we show:

\begin{thmintro}\label{thm: ellipticity + cohomological fctr}
Let $X$ be a complex manifold and $\cL_{p,q}(X)$ its Schweitzer complex. 
\begin{enumerate}
\item\label{thm: ellipticity} The complex $\cL_{p,q}^\Cdot(X)$ is elliptic.
\item\label{thm: cohomological fctr} For every $p,q,k\in\Z$, the association $X\mapsto H^k(\cL_{p,q}^\Cdot(X))$ defines a cohomological functor in the sense of \cite{stelzig_structure_2018, stelzig_linear_2021}.
\end{enumerate}
\end{thmintro}
From now on, let us assume $X$ to be compact and of dimension $n$. Then we obtain:
\begin{corintro}\label{corintro: Finite dimensionality}
The dimensions $s_{p,q}^k(X):=\dim H^k(\cL_{p,q}(X))$ are finite.
\end{corintro}
This has also been shown by Demailly \cite[Thm. 12.4]{demailly_complex_nodate}, using a different argument. Further, we have
\begin{corintro}\label{corintro: Perfect pairing}
For $X$ connected, wedge product and integration induce a perfect pairing 
\[H^k(\cL_{p,q}(X))\times H^{2n-k-1}(\cL_{n-p+1,n-q+1}(X))\to\C.\]
\end{corintro}

\begin{corintro}\label{corintro: Semi-continuity}
The numbers $s_{p,q}^k(X_t)$ vary upper semi-continuously in holomorphic families.
\end{corintro}
We will illustrate that this last result gives new restrictions on the $E_1$-isomorphism-type under small deformations. In particular, one obtains semi-continuity results even for classical objects such as certain differentials in the Fr\"olicher spectral sequence.

By ellipticity, we may apply the Atiyah-Singer index theorem and obtain equalities between the Euler-characteristics $\chi_{p,q}(X)$ of $\cL_{p,q}(X)$ and certain expressions $td_{p,q}(X)$ in characteristic numbers.
\begin{thmintro}\label{thm: ABC=Hirz}
The relations $\chi_{p,q}(X)=td_{p,q}(X)$ are equivalent to the Hirzebruch-Riemann-Roch relations.
\end{thmintro}
Theorem \ref{thm: ABC=Hirz} is in accord with a conjecture made in \cite{stelzig_linear_2021}, stating that any universal linear relation between cohomological invariants and Chern numbers of compact complex manifolds of a given dimension is a consequence of the Hirzebruch-Riemann-Roch relations.

\subsection*{Acknowledgments} I thank Rui Coelho, Alexander Milivojevi\'c, Nicola Pia and Giovanni Placini for many inspiring conversations.

\section*{Proofs of the main results}
\subsection*{The Schweitzer complex} Throughout, we fix an $n$-dimensional complex manifold $X$ and we denote by $\cA_X^{p,q}$ the sheaf of smooth complex valued $(p,q)$-forms. Given any fixed pair of integers $p,q\in\Z$, Schweitzer and Demailly \cite{schweitzer_autour_2007, demailly_complex_nodate}, define a simple complex $\cL_{p,q}^\Cdot$ of locally free sheaves as follows:
\begin{align*}
\cL_{p,q}^k&:=\bigoplus_{\substack{r+s=k\\r<p,s<q}} \cA^{r,s}_X&\text{if }k\leq p+q-2\\
\cL_{p,q}^k&:=\bigoplus_{\substack{r+s=k+1\\r\geq p,s\geq q}} \cA^{r,s}_X&\text{if }k\geq p+q-1
\end{align*}
with differential $d_{\cL}$ given by:

\[
\cdots\overset{\pr\circ d}{\longrightarrow}\cL_{p,q}^{p+q-3}\overset{\pr\circ d}{\longrightarrow}\cL_{p,q}^{p+q-2}\overset{\del\delbar}{\longrightarrow}\cL_{p,q}^{p+q-1}\overset{d}{\longrightarrow}\cL_{p,q}^{p+q}\overset{d}{\longrightarrow}\cdots
\]
We illustrate which components of $\cA_X$ contribute to $\cL_{p,q}^\Cdot$ for $n=3$, $(p,q)=(2,1)$:
\[
\begin{tabular}{cccc}
%\midrule
&&$\cA^{2,3}_X$&$\cA^{3,3}_X$\\
%\midrule
&&&\\
&&$\cA^{2,2}_X$&$\cA^{3,2}_X$\\%\hline
&&&\\
%\midrule
&&$\cA^{2,1}_X$&$\cA^{3,1}_X$\\%\hline
&&&\\
%\midrule
$\cA^{0,0}_X$&$\cA^{1,0}_X$&&\\%\hline
%\midrule
\end{tabular}
\]
By construction, one has
\[
H^{p,q}_{BC}(X)=H^{p+q-1}(\cL_{p,q}^\Cdot(X)):=\mathbb{H}^{p+q-1}(\cL_{p,q}^\Cdot)
\]
and
\[
H^{p,q}_{A}(X)=H^{p+q-2}(\cL_{p,q}^\Cdot(X)):=\mathbb{H}^{p+q-2}(\cL_{p,q}^\Cdot)
\]
We will mainly be interested in the differential complex of global sections $\cL_{p,q}^\Cdot(X)$.\footnote{Up to a shift in total degree, this complex is denoted $S_{p,q}(X)$ in \cite{milivojevic_stelzig_2022}. Here we follow the indexing convention used in \cite{schweitzer_autour_2007, demailly_complex_nodate}.}

\subsection*{Cohomological functors} We recall \cite{stelzig_structure_2018, stelzig_linear_2021} that a cohomological functor on the category of compact\footnote{If we consider all complex manifolds, one should modify this definition by lifting the dimension restriction on the target of $H$, but imposing the condition that $\dim H(Z)$ is finite dimensional for any (bounded) indecomposable double complex.} complex manifolds is a functor to the category of finite dimensional vector spaces which factors as $H\circ A$, where $A: X\mapsto A_X:=(\cA_X(X),\del,\delbar)$ is the Dolbeault double complex functor and $H$ is a linear functor which vanishes on `squares', i.e. double complexes of the form
\[
\begin{tikzcd}
\C\ar[r,"\sim"]&\C\\
\C\ar[u,"\sim"]\ar[r,"\sim"]&\C\ar[u,"\sim"]
\end{tikzcd}.
\]
It is clear from the construction that $\cL_{p,q}(X)$ depends only on $A_X$. For any double complex $A$, denote by $\cL_{p,q}(A)$ the Schweitzer complex formed using $A$ instead of $\cA_X(X)$ and denote $H_{S_{p,q}}^k(A):=H^k(\cL_{p,q}(A))$. It is then clear that $H_{S_{p,q}}^k(A)$ is a linear functor (in particular it commutes with direct sums of double complexes) and what remains to show is the following:
\begin{lem}
For any square $Sq$, we have $H_{S_{p,q}}^k(Sq)=0$ for all $p,q,k\in\Z$.
\end{lem}
\begin{proof}
Any square $Sq$ is concentrated in four bidegrees: 
\[
S:=\supp Sq:=\{(r,s),(r+1,s),(r,s+1),(r+1,s+1)\}.
\]
The structure of $\cL_{p,q}(Sq)$ (and hence $H_{S_{p,q}}(Sq)$) depends on the relative position of $S$ and $T=T_A\cup T_B$ where
\begin{align*}
T_A&:=\{(a,b)\in\Z^2\mid a<p,b<q\}\\
T_B&:=\{(a,b)\in\Z^2\mid a\geq p,b\geq q\}.
\end{align*}
There are three possibilities: $\#(S\cap T)\in\{0,2,4\}$, surveyed in the following table:
\begin{center}
\begin{tabular}{ccc}
\toprule
Case 1: $S\cap T=\emptyset$& Case 2: $\#(S\cap T)=2$ & Case 3: $\#(S\cap T)=4$\\
\cmidrule(r){1-1}\cmidrule(lr){2-2}\cmidrule(l){3-3}
$\img{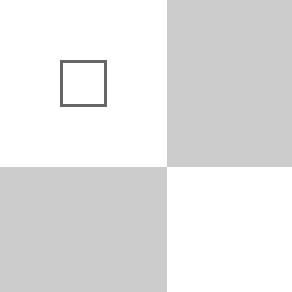}$&$\img{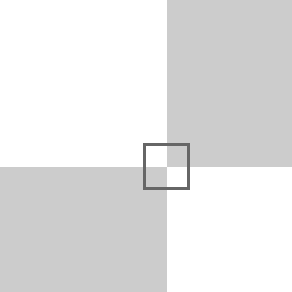}$ or $\img{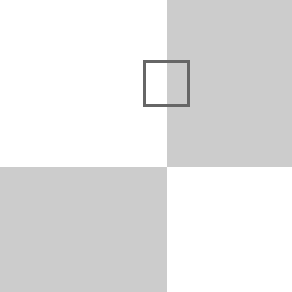}$&$\img{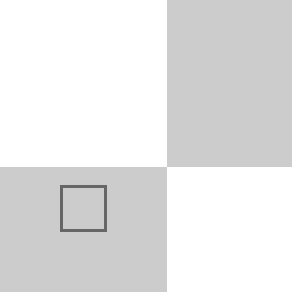}$\\
\cmidrule(r){1-1}\cmidrule(lr){2-2}\cmidrule(l){3-3}
$\cL_{p,q}(Sq)=0$&$\cL_{p,q}(Sq)=\C\overset{\simeq}{\longrightarrow}\C$&$\cL_{p,q}(Sq)=\C\hookrightarrow\C^2\twoheadrightarrow\C$\\
\bottomrule
\end{tabular}
\end{center}
In each case, clearly $H_{S_{p,q}}^k(Sq)=H^k(\cL_{p,q}(Sq))=0$ for all $p,q,k$.
\end{proof}

Since for any compact $X$, there exists a decomposition $A_X=A_X^{sq}\oplus A_X^{zig}$ where $A_X^{sq}$ is a direct sum of squares and $A_X^{zig}$ is finite dimensional \cite{stelzig_structure_2018}, this implies Corollary \ref{corintro: Finite dimensionality}. 

\begin{proof}[Proof of Corollary \ref{corintro: Perfect pairing}]
We recall from \cite[Cor. 20]{stelzig_structure_2018} that for any cohomological functor $H$, the integration pairing induces an isomorphism $H(A_X)\cong H(DA_X)$, where $DA_X$ denotes the dual double complex as in \cite{stelzig_structure_2018}. The bigraded components of the dual complex are given by $(DA_X)^{p,q}=(A_X^{p,q})^{\vee}$ and the differentials (up to sign) by pullback with the differentials of $A_X$. Therefore, we have $\cL_{p,q}^k(DA_X)=(\cL_{n-p+1,n-q+1}^{2n-k-1}(A_X))^\vee$. But cohomology of a complex of vector spaces commutes with dualization, so
\[
H^k(\cL_{p,q}(A_X))\cong H^k(\cL_{p,q}(DA_X))=H^{2n-k-1}(\cL_{n-p+1,n-q+1}(X))^\vee.
\]\end{proof}
\subsection*{Ellipticity}	Let $\pi:TX^{\vee}\longrightarrow X$ denote the projection map of the cotangent bundle and for any $\xi=\xi_x\in TX^\vee_x$ denote by $(L^\Cdot,\sigma):=(\pi^*\cL_{p,q}^\Cdot(X)_\xi,\sigma(d_{\cL_{p,q}})(\xi))$ the symbol complex (c.f. \cite{atiyah-bott_lefschetz_67}). To show ellipticity, we have to prove:

\begin{lem}
For any $\xi\neq 0$, the symbol complex $(L^\Cdot,\sigma)$ is exact.
\end{lem}
\begin{proof} As is well-known, we have $\sigma(\del)(\xi)=\xi^{1,0}\wedge\_$ and $\sigma(\delbar)(\xi)=\xi^{0,1}\wedge\_$, where the superscripts mean projection to the corresponding bidegree. Without loss of generality, we may pick a local coordinate system $z_1,...,z_n$ around $x$ such that $\xi=dz_1+d\bar{z}_1$. 
	
	Let us check exactness of the following part of the complex. Exactness at stages of lower or higher degree is only notationally more cumbersome.
	\[
	\begin{array}{c}A_{X,x}^{p-3,q-1}\\\oplus\\ A_{X,x}^{p-2,q-2}\\\oplus\\ A_{X,x}^{p-1,q-3}\end{array}\longrightarrow \begin{array}{c}A_{X,x}^{p-2,q-1}\\\oplus\\ A_{X,x}^{p-1,q-2}\end{array}\longrightarrow A_{X,x}^{p-1,q-1}\longrightarrow A_{X,x}^{p,q}\longrightarrow\begin{array}{c}A_{X,x}^{p,q+1}\\\oplus\\ A_{X,x}^{p+1,q}\end{array}\longrightarrow \begin{array}{c}A_{X,x}^{p,q+2}\\\oplus\\ A_{X,x}^{p+1,q+1}\\\oplus\\ A_{X,x}^{p,q+2}\end{array} .
	\]
	The first map is given by 
	 \[
	 \sigma^{p+q-4}: \omega^{p-3,q-1}+\omega^{p-2,q-2}+\omega^{p-1,q-3}\longmapsto \begin{array}{l} \xi^{1,0} \omega^{p-3,q-1}+\xi^{0,1} \omega^{p-2,q-2}+\\
	 \xi^{1,0} \omega^{p-2,q-2}+\xi^{0,1} \omega^{p-1,q-3}\end{array}.	 
	 \]
	The second map is given by
	\[
	\sigma^{p+q-3}:\omega^{p-2,q-1}+\omega^{p-1,q-2}\longmapsto \xi^{1,0} \omega^{p-2,q-1}+\xi^{0,1} \omega^{p-1,q-2}.
	\] 
	The third map is given by
	\[
	\sigma^{p+q-2}: \omega\longmapsto \xi^{1,0}\xi^{0,1}\omega.
	\]
	The fourth and fifth maps are given by $
	\omega\longmapsto \xi \omega$.

	Now assume $\omega=\omega^{p-2,q-1}+\omega^{p-1,q-2}\in\ker\sigma^{p+q-3}$. Write 
	\begin{align*}
	\omega^{p-2,q-1}&=\xi^{1,0}\xi^{0,1} \omega_A+\xi^{1,0} \omega_B+\xi^{0,1} \omega_C+\omega_D\\
	\omega^{p-1,q-2}&=\xi^{1,0}\xi^{0,1} \omega'_A+\xi^{1,0} \omega'_B+\xi^{0,1} \omega'_C+\omega'_D,	
	\end{align*}
	with $\omega_A,\omega'_A,...$ having no summand which contains a factor of $\xi^{1,0}$ or $\xi^{0,1}$. Then $\sigma(\omega)=0$ translates into the three equations
	\begin{align*}
	\xi^{1,0}\xi^{0,1}\omega_C+\xi^{0,1}\xi^{1,0}\omega_B'&=0\\
	\xi^{1,0}\omega_D=\xi^{0,1}\omega'_D=0.
	\end{align*}
	The second and third equation imply $\omega_D=\omega'_D=0$ and the first implies $\omega_C=\omega_B'$. Hence, defining $\eta\in L_x^{p+q-4}$ as follows
	\begin{align*}
	\eta^{p-3,q-1}&:=\xi^{0,1}\omega_A + \omega_B \\
	\eta^{p-2,q-2}&:= \omega_C\\
	\eta^{p-1,q-3}&:= -\xi^{1,0}\omega_A'+\omega'_C,
	\end{align*}
	one obtains $\sigma(\eta)=\omega$.
	
	Now, let $\omega\in L^{p+q-2}_x=A_{X,x}^{p-1,q-1}$. Write $\omega=\xi^{1,0}\xi^{0,1}\omega_A+\xi^{1,0}\omega_B+\xi^{0,1}\omega_C+\omega_D$ as before. Then $0=\sigma(\omega)=\xi^{1,0}\xi^{0,1}\omega$ implies $\omega_D=0$. Hence, defining $\eta\in L_x^{p+q-3}$ by
	\begin{align*}
	\eta^{p-2,q-1}:=\omega_B+ \frac 1 2 \xi^{0,1}\omega_A \\
	\eta^{p-1,q-2}:=\omega_C -\frac 12 \xi^{1,0}\omega_A,
	\end{align*} 
	we get $\sigma(\eta)=\omega$.

	Next, let $\omega\in L^{p+q-1}_x=A_{X,x}^{p,q}$. Write $\omega=\xi^{1,0}\xi^{0,1}\omega_A+\xi^{1,0}\omega_B+\xi^{0,1}\omega_C+\omega_D$ as above. Then sorting $0=\sigma(\omega)=\xi^{1,0}\omega+\xi^{0,1}\omega$ by bidegree yields 
	\begin{align*}
	\xi^{1,0}\xi^{0,1}\omega_C+\xi^{1,0}\omega_D&=0\\
	\xi^{0,1}\xi^{1,0}\omega_B+\xi^{0,1}\omega_D&=0.
	\end{align*}
	In particular, $\omega_D=-\xi^{0,1}\omega_C$ and therefore $0=\omega_D=\omega_C=\omega_B$. Thus, $\omega=\sigma(\omega_A)$.	Finally, let $\omega=\omega^{p,q+1}+\omega^{p+1,q}\in A_{X,x}^{p,q+1}\oplus A_{X,x}^{p+1,q}=L_x^{p+q}$ s.t. $0=\sigma(\omega)=\xi\wedge\omega$. Again, write 
	\begin{align*}
	\omega^{p,q+1}&=\xi^{1,0}\xi^{0,1}\omega_A+\xi^{1,0}\omega_B+\xi^{0,1}\omega_C+\omega_D\\
	\omega^{p+1,q}&=\xi^{1,0}\xi^{0,1}\omega'_A+\xi^{1,0}\omega'_B+\xi^{0,1}\omega'_C+\omega'_D.
	\end{align*}
	Then $\sigma(\omega)=0$ translates into
	\begin{align*}
		0=&\xi^{0,1}\xi^{1,0}\omega_B+\xi^{0,1}\omega_D\\
		0=&\xi^{1,0}\xi^{0,1}\omega_C+\xi^{1,0}\omega_D+\xi^{0,1}\xi^{1,0}\omega'_B+\xi^{0,1}\omega'_D\\
		0=&\xi^{1,0}\xi^{0,1}\omega'_C+\xi^{1,0}\omega'_D.
	\end{align*}
	This implies $\omega_D=\omega_B=\omega'_C=\omega'_D=0$ and $\omega_C=\omega'_B$. Defining 	
	\[ 
	\eta:=-\xi^{1,0}\omega_A+\omega_C+\xi^{0,1}\omega'_A\in L^{p+q-1}=A^{p,q}_{X,x},
	\]
	we have $\sigma(\eta)=\omega$.
\end{proof}

%\begin{rem}
%	Note that $\del\delbar$ is of order $2$, while all other differentials are of order $1$. Thus, when $p+q\neq k, k-1$, the cohomology $H^k(\cL_{p,q}^\Cdot(X))$ can be computed as the kernel of the operators $\Delta_{p,q}^k:=(d_{\cL_{p,q}}^k)^\ast d_{\cL_{p,q}}^k+ (d_{\cL_{p,q}}^{k-1})(d_{\cL_{p,q}}^{k-1})^\ast$ 
%\end{rem}

\begin{rem}
Ellipticity immediately gives a second proof of Corollary \ref{corintro: Finite dimensionality}, see \cite[Prop. 6.5.]{atiyah-bott_lefschetz_67}
\end{rem}
%\begin{cor}
%The spaces $H^k(\cL^\Cdot_{p,q}(X))$ are finite-dimensional. 
%\end{cor}

\subsection*{Small deformations}
Let us prove Corollary \ref{corintro: Semi-continuity}:

%\begin{cor}
%	The dimensions of the spaces $H^k(\cL_{p,q}^\Cdot(X))$ vary upper semi-continuously in families.
%\end{cor}

\begin{proof}
Pick a hermitian metric. For any $k\neq p+q$, the order of $d_{\cL_{p,q}}^k$ is $1$ and thus for $k\neq p+q,p+q-1$, the operators $\Delta_{p,q}^k:=(d_{\cL_{p,q}}^k)^\ast d_{\cL_{p,q}}^k+ d_{\cL_{p,q}}^{k-1}(d_{\cL_{p,q}}^{k-1})^\ast$ are elliptic and their kernel is isomorphic to $H^k(\cL_{p,q}^\Cdot(X))$. They vary smoothly in families and hence the result follows. For $k=p+q,p+q-1$, i.e. for Bott-Chern and Aeppli cohomology, the Corollary is known. (In that case, $\Delta_{p,q}^k$ as defined above is not elliptic, but an appropriate modification of it is, see  \cite[sect. 6]{atiyah-bott_lefschetz_67} for a general statement or \cite{schweitzer_autour_2007, demailly_complex_nodate} for an explicit construction in this case.)
\end{proof}
As an example we use this result to obtain semicontinuity properties for a classical object: the Fr\"olicher spectral sequence \cite{frolicher_relations_1955}. We denote by $e_r^{p,q}:=\dim E_r^{p,q}$ the dimensions of the bigraded pieces on the $r$-th page of this spectral sequence.  The dimensions on the first page $e_1^{p,q}=h^{p,q}_{\delbar}$ are known to behave upper semi-continuously, but  for later pages this is false in general. However, denoting by $FD^{p,q}:=e_1^{p,q}-e_\infty^{p,q}$, we may show:
\begin{cor}
Let $X_t$ be a small deformation with $X_0=X$ compact of dimension $n$. Then for $t$ sufficiently close to zero, there are inequalities
\[FD^{0,1}(X_0)\geq FD^{0,1}(X_t)\quad \text{and}\quad FD^{0,n-1}(X_0)\geq FD^{0,n-1}(X_t).\] 
%\[
%\del_1^{0,1}(t):H^{0,1}(X_t)\to H^{1,1}(X_t)
%\]
%the Fr\"olicher differential on the first page in degree $(0,1)$. Then the number 
%\[
%I(t):=\dim \im \del_1^{0,1}(t)
%\] 
%varies upper semi-continuously. In particular, if $\del_1^{0,1}(0)=0$, then also $\del_1^{0,1}(t)=0$ for all nearby $t$.
\end{cor}
\begin{proof}
In order to avoid heavy notation which obscures the idea of the proof, we do the proof in the case of $3$-folds, where it is easier to draw all necessary diagrams. We use \cite{khovanov_faithful_2020, stelzig_structure_2018} that $A_X$ may be decomposed into a direct sum of indecomposable double complexes, namely infinitely many squares and finitely many zigzags. The multiplicities $\mult_Z(A_X)=:\#Z(A_X)$ of these summands are an isomorphism invariant of $A_X$. Furthermore, for each zigzag, its images under the involutions $\tau:(p,q)\mapsto(q,p)$ `flipping along the diagonal' and $\sigma:(p,q)\mapsto (n-p,n-q)$ `flipping along the main antidiagonal' appear with the same multiplicity, so, denoting by $Z'$ the sum of the elements in the $\langle\sigma,\tau\rangle$-orbit of $Z$ we may write $\#Z'=\#Z$. Finally, certain zigzags always have zero multiplicity (`only dots in the corners'), see \cite[Ch. 4]{stelzig_structure_2018}. With this understood we compute:
\[
b_1(X):=2\cdot\#\img[.8]{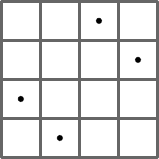}\,(A_X) + \#\img[.8]{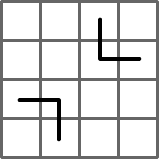}\,(A_X)\quad\text{and}\quad h_{BC}^{0,1}(X)=\#\img[.8]{S110_3}\,(A_X),
\]
whereas
\[
e_{\infty}^{0,1}(X)=\#\img[.8]{S110_3}\,(A_X) + \#\img[.8]{S100_3}\,(A_X).
\]
%
%\[
%h_{\delbar}^{0,1}=\#\img{S110_3} + \#\img{S100_3} + \#\img{S1101_3} + \img{S1201_3}
%\]
%\[
%e_{\infty}^{0,1}=\#\img{S110_3} + \#\img{S100_3}
%\]
Therefore $e^{0,1}_\infty=b_1-h_{BC}^{0,1}$. Since $b_1$ stays constant in families and $h_{BC}^{0,1}$ behaves upper semi-continuously, the first inequality, $FD^{0,1}(X) \geq FD^{0,1}(X_t)$, follows. For the second one, we compute
\[
b_3(X)=2\cdot\#\img[.8]{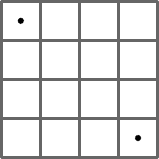}\,(A_X) + 2\cdot\#\img[.8]{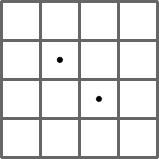}\,(A_X) + 2\cdot\#\img[.8]{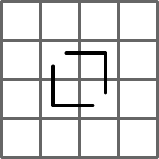}\,(A_X)
\]
and, writing as before $s_{p,q}^k(X):=\dim H^k(\cL_{p,q}(X))$,
\begin{align*}
s_{1,0}^2(X)=&\#\img[.8]{S330_3}\,(A_X) + 2\cdot\#\img[.8]{S321_3}\,(A_X) +2\cdot\#\img[.8]{S311_3}\,(A_X)~+\\
&\#\img[.8]{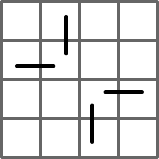}\,(A_X) + \#\img[.8]{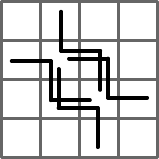}\,(A_X).
\end{align*}
On the other hand, $FD^{0,2}(X)=h^{0,2}(X)-e_\infty^{0,2}(X)$ counts the dimensions of all differentials in the Fr\"olicher spectral sequence starting at $(0,2)$, i.e.
\[
FD^{0,2}(X)=\#\img[.8]{S1102_3}\,(A_X) + \#\img[.8]{S1202_3}\,(A_X).
\]
Therefore,
\[
FD^{0,2}=h_{BC}^{0,3}+s_{1,0}^2-b_3,
\]
which is upper semi-continuous since the first two terms are and $b_3$ is constant. \end{proof}

\begin{rem}
	The proof also shows that $e_\infty^{0,1}$ behaves lower semi-continuously. The result should be compared to the fact that $e_\infty^k=\bigoplus_{p+q=k}e_\infty^{p,q}=b_k$ is constant along deformations and the total Fr\"olicher defect $FD^{k}=\bigoplus_{p+q=k}FD^{p,q}=h_{\delbar}^k-b_k$ is upper semi-continuous.
\end{rem}
%\begin{rem}
%It would be interesting to work out whether this yields new restrictions on the possible $E_1$-isomorphism types of small deformations of a given compact complex manifold.
%\end{rem}

\subsection*{Index formulae}
Denote by $\chi_p(X):=\sum_q(-1)^qh^{p,q}_{\delbar}(X)$ and by $td_p(X):=\int_XTd(X)ch(\Omega^p)$, where $Td(X)$ denotes the Todd class and $ch$ the Chern character. The Hirzebruch-Riemann-Roch relations can be expressed as $\chi_p(X)=td_p(X)$ for all $p\in\Z$, see \cite{atiyah-singer_index_63}.

Similarly, let us denote by $\chi_{p,q}(X):=\sum_{k}(-1)^k\dim H^k(\cL_{p,q}^\Cdot(X))$ the analytical index of $\cL_{p,q}^\Cdot$ and by $td_{p,q}(X)$ its topological index, which is the pairing of the fundamental class $[X]$ with a universal characteristic class which depends only on the $K$-theory class $[\cL_{p,q}^\Cdot]:=\sum_{k}(-1)^k [\cL_{p,q}^k]\in K(X)$. The Atiyah-Singer index theorem \cite{atiyah-singer_index_63, atiyah-singer_index_86} then yields:
 
\begin{thm}[ABC index formulae]\label{thm: ABC index formulae} 
For any compact complex manifold $X$ and $p,q\in\Z$,
\[
\chi_{p,q}(X)=td_{p,q}(X).
\]	
\end{thm}

\begin{rem}
Strictly speaking, \cite{atiyah-singer_index_63, atiyah-singer_index_86} treat only elliptic complexes where all operators have order one, which is not the case for $\cL_{p,q}(X)$. However, the validity of Theorem \ref{thm: ABC index formulae} also follows a posteriori from Theorem \ref{thm: ABC=Hirz}.
\end{rem}

%\begin{thm}
%	The ABC-index formulae are a equivalent to the Hirzebruch Riemann Roch relations $\chi_p(X)=td_p(X)$ for all $p\geq \Z$.
%\end{thm}
\begin{proof}[Proof of Theorem \ref{thm: ABC=Hirz}]
	
	Recall from \cite{demailly_complex_nodate, schweitzer_autour_2007} that there are subcomplexes $(\cS^\Cdot_p,\del)$ and $(\bar{\cS}^\Cdot_q,\delbar)$ of $\cL_{p,q}$ defined as follows (if $p,q\geq 1$):
	
	\[
	\cS^k_p:=\begin{cases}\Omega_X^k&\text{if } 0\leq k\leq p-1\\
						0&\text{else.}
			\end{cases}
			\qquad
	\bar{\cS}^k_q:=\begin{cases}\bar\Omega_X^k&\text{if } 0\leq k\leq q-1\\
	0&\text{else.}
	\end{cases}
	\] If $p=0$ or $q=0$, one sets instead $\cS^0_p=\C$, resp. $\bar{\cS}^0_q=\C$ and all other components equal to $0$. Set $\cS_{p,q}^k:=\cS^k_p+\bar{\cS}^k_q$, the sum being direct except for $k=0$. It is then known \cite{demailly_complex_nodate, schweitzer_autour_2007} that $\cS_{p,q}\hookrightarrow\cL_{p,q}$ is a quasi-isomorphism. Further, there is a short exact sequence of complexes
	\[
	0\longrightarrow \C\longrightarrow\cS_p^\Cdot\oplus \bar\cS_q^\Cdot\longrightarrow \cS_{p,q}^\Cdot\longrightarrow 0.
	\]
	For a complex $C$ of abelian sheaves on $X$, denote by $\chi(C):=\sum_k (-1)^k\dim \HH^k(C)$ its hypercohomology Euler characteristic. Then, using $\chi_p=(-1)^n\chi_{n-p}$, there is an equality of Euler characteristics:
	
	\begin{align*}
	\chi_{p,q}(X)=\chi(\cS_{p,q}^\Cdot)&=\chi(\cS^\Cdot_p) + \chi(\bar{\cS_q}^\Cdot)- \chi(\C)\\
	&=\sum_{k=0}^{p-1}(-1)^k\chi_k + \sum_{k=0}^{q-1}(-1)^k\chi_k-\sum_{k=0}^n(-1)^{k}\chi_k\\
	&=\sum_{k=0}^{p-1}(-1)^{k}\chi_k +\sum_{k=n-q}^{n}(-1)^k\chi_k-\sum_{k=0}^n(-1)^{k}\chi_k=\sum_{k=p}^{n-q}(-1)^{k+1}\chi_k
	\end{align*}
	 
%	\[
%	\Oh_X\longrightarrow \Omega^1\longrightarrow \cdots\longrightarrow \omega^{p-1}\longrightarrow 0
%	\]	
	To identify the characteristic number expressions, we instead identify the $K$-theory classes of the relevant complexes, using only the relation $(\cA^{p,q}_X)^\vee\cong \cA^{n-q,n-p}_X$.
	Assume for simplicity $p+q\leq n$ and $p>q$, the other cases are similar. The following chain of equalities in $K(X)$ might be easier to follow with the following picture in mind, which illustrates the case $n=3$, $(p,q)=(2,1)$:
	\[
	\begin{array}{|c|c|c|c|}
	\hline
	&&+&-\\\hline
	&&-&+\\\hline
	&&+&-\\\hline
	+&-&&\\\hline
	\end{array}
	\rightsquigarrow
	\begin{array}{|c|c|c|c|}
	\hline
	\phantom{+}&\phantom{+}&+&\\\hline
	&&-&\\\hline
	&&+&-\\\hline
	&&&\\\hline
	\end{array}
	\rightsquigarrow
	\begin{array}{|c|c|c|c|}
	\hline
	\phantom{+}&\phantom{+}&+&\\\hline
	&&-&\phantom{+}\\\hline
	&&+&\\\hline
	&&-&\\\hline
	\end{array}
	\]
		
	\begin{align*}
	\cL_{p,q}^\Cdot=\sum_{k\in\Z} (-1)^k\cL_{p,q}^{k}&=\sum_{k\leq p+q-2}(-1)^k\sum_{\substack{r+s=k\\r<p,s<q}} \cA_X^{r,s} + \sum_{k\geq p+q-1}(-1)^k\sum_{\substack{r+s=k+1\\r\geq p, s\geq q}} \cA_X^{r,s}\\
	&=\sum_{r=0}^{p-1}\sum_{s=0}^{q-1}(-1)^{r+s}\cA_X^{r,s} + \sum_{r=p}^n\sum_{s=q}^n (-1)^{r+s-1}\cA_X^{r,s}\\
	&=\sum_{r=n-q+1}^n\sum_{s=n-p+1}^n (-1)^{r+s}\cA_X^{r,s} + \sum_{r=p}^n\sum_{s=q}^n (-1)^{r+s-1}\cA_X^{r,s}\\
	&= \sum_{r=p}^{n-q} \sum_{s=q}^n(-1)^{r+s-1}\cA_X^{r,s} + \sum_{r=n-q+1}^n\sum_{s=q}^{n-p}(-1)^{r+s-1}\cA_X^{r,s}\\
	&= \sum_{r=p}^{n-q}\sum_{s=0}^n (-1)^{r+s-1}\cA_X^{r,s}= \sum_{r=p}^{n-q}(-1)^{r+1} \cA_X^{r,\Cdot}.
	\end{align*}
	And therefore:
	\[
	td_{p,q}=\sum_{r=p}^{n-q}(-1)^{r+1} td_p,
	\]
	which implies the theorem since $\chi_{p,q}=td_{p,q}$ for all $p,q$ if and only if $\chi_p= td_p$ for all $p$.
\end{proof}

\begin{rem}In particular, we have shown the following relation between Dolbeault and Schweitzer cohomologies:
\[\chi_{p,q}(X)=\sum_{k=p}^{n-q}(-1)^{k+1}\chi_p(X).
\]
\end{rem}

\bibliographystyle{acm}

\end{document}